\documentclass[12pt]{article}
\usepackage[paperwidth=8.5in, paperheight=11in, top=1in, bottom=1in, left=.5in, right=.5in]{geometry}

\usepackage{amsmath,amsfonts,amssymb,epsfig,graphicx}
\usepackage{xspace,multicol,paracol,caption}
\usepackage{theorem}
\usepackage{color}

\newtheorem{theorem}{Theorem}[section]

\newtheorem{lemma}[theorem]{Lemma}
\newtheorem{question}{Question}

\numberwithin{equation}{section}

\newenvironment{proof}[1][Proof]{\begin{trivlist}
\item[\hskip \labelsep {\bfseries #1}]}{\end{trivlist}}

\newcommand{\qed}{\hfill \ensuremath{\Box}}

\begin{document}

\thispagestyle{empty}
\title{\textbf{On Rainbow Cycles and Proper Edge Colorings of Generalized Polygons}}

\author{\textbf{Matt Noble} \\
Department of Mathematics and Statistics \\
Middle Georgia State University\\
matthew.noble@mga.edu}
\date{}
\maketitle

\begin{abstract}

An edge coloring of a simple graph $G$ is said to be \textit{proper rainbow-cycle-forbidding} (PRCF, for short) if no two incident edges receive the same color and for any cycle in $G$, at least two edges of that cycle receive the same color.  A graph $G$ is defined to be \textit{PRCF-good} if it admits a PRCF edge coloring, and $G$ is deemed \textit{PRCF-bad} otherwise.  In recent work, Hoffman, et al. study PRCF edge colorings and find many examples of PRCF-bad graphs having girth less than or equal to 4.  They then ask whether such graphs exist having girth greater than 4.  In our work, we give a straightforward counting argument showing that the Hoffman-Singleton graph answers this question in the affirmative for the case of girth 5.  It is then shown that certain generalized polygons, constructed of sufficiently large order, are also PRCF-bad, thus proving the existence of PRCF-bad graphs of girth 6, 8, 12, and 16.\\[5pt]    

\end{abstract}

\section{Introduction}

All graphs will be finite and simple.  The notation and terminology used will be for the most part standard, and for clarification, one may consult any introductory text.  A few graph parameters deserve special mention, as they will be central to our discussion.  For a graph $G$, define the \textit{girth} of $G$, notated as $g(G)$, to be the minimum length of a cycle in $G$.  For $a,b \in V(G)$, the \textit{distance} $d(a,b)$ is the minimum length of a path beginning at $a$ and terminating at $b$.  The \textit{diameter} of a connected graph $G$, which we will notate as $diam(G)$, is equal to the maximum distance taken over all pairs of vertices in $G$.  

For a graph $G$ and a set of colors $C$, let $\varphi: E(G) \to C$ be an edge coloring.  A subgraph $H$ of $G$ is said to be \textit{rainbow} (with respect to $\varphi$) if every edge of $H$ receives a different color.  Define $\varphi$ to be \textit{proper rainbow-cycle-forbidding} if no two incident edges of $G$ receive the same color (that is, $\varphi$ is proper), and no cycle in $G$, regardless of length, is rainbow.  Define a graph $G$ to be \textit{PRCF-good} if it admits a PRCF edge coloring, and deem $G$ to be \textit{PRCF-bad} otherwise.  Just to be clear, note that in a PRCF edge coloring of $G$, there is no stipulation on how many colors are actually used.  For $G$ to be PRCF-good, it may be that $G$ has a PRCF edge coloring using a number of colors equal to the chromatic index $\chi'(G)$, or it may be that $G$ has a PRCF edge coloring using some other number of colors.  However, for $G$ to be PRCF-bad, no matter the number of colors used in a proper edge coloring, $G$ is guaranteed to have at least one rainbow cycle.    

In recent work \cite{hoffmanetal}, Hoffman, et al. study PRCF edge colorings and produce numerous examples of PRCF-bad graphs -- some of these found by computer search, some generated through constructive methods, and some found through nothing more than pencil and paper efforts.  The complete graph $K_3$ is a trivial example of a PRCF-bad graph, as is any $G$ having $K_3$ as a subgraph.  It is fairly easy to see that the complete bipartite graph $K_{2,4}$ is also PRCF-bad, and we invite curious readers to take a moment to verify this fact for themselves.  A this point, it should be noted, however, that every PRCF-bad graph given in \cite{hoffmanetal} has girth equal to 3 or 4.  This begs the question, does there exist a PRCF-bad graph of girth 5?  More generally, do there exist a PRCF-bad graphs of arbitrarily large girth?

We will in our work prove that the Hoffman-Singleton graph answers the above question in the affirmative in the case of girth 5.  We will then show that the family of generalized polygons contains PRCF-bad graphs $G$ with $g(G) = y$ for each $y \in \{6,8,12,16\}$.  Such graphs are well-studied in the literature, and although their construction can be somewhat laborious to explicitly describe, our argument showing that a generalized polygon of sufficiently large order is indeed PRCF-bad employs nothing more than elementary counting techniques.  We will close with a few questions and observations which may hopefully guide future efforts to answer the above question in general.

\section{The Hoffman-Singleton Graph}

The Hoffman-Singleton graph was discovered by the eponymous A. J. Hoffman and R. R. Singleton and introduced in \cite{hoffmansingleton}.  Although it will not be fully described here, there are numerous sources detailing its construction (see \cite{godsil}, page 92 or see \cite{mathworld} for a visual depiction).  The graph is very well-studied and many of its parameters have been established.  Here, will we be concerned with the following aspects of the Hoffman-Singleton graph.  It has order 50, it is 7-regular, it has girth 5, and it has diameter 2.

The diameter and girth of the Hoffman-Singleton graph are important.  Let $G$ denote the Hoffman-Singleton graph, and consider the path $P_4$ appearing as a subgraph of $G$.  Let the vertex set of this path be $\{v_1,v_2,v_3,v_4\}$, labeled as in the diagram below.

\begin{center}
\includegraphics[scale=0.65]{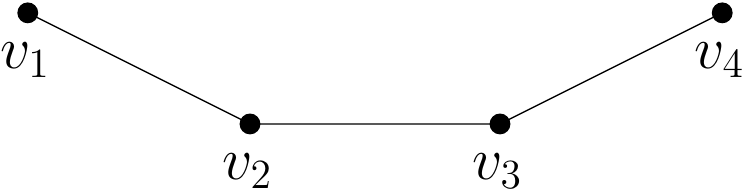}
\end{center}

Since $g(G) = 5$, $v_1v_4 \not \in E(G)$.  As $diam(G) = 2$, there exists $x \in V(G)$ adjacent to both $v_1$ and $v_4$.  Furthermore, this vertex $x$ must be unique, as the existence of another vertex $y$ adjacent to $v_1$ and $v_4$ would induce a 4-cycle in $G$ on vertices $v_1,x,v_4,y$.  Putting these observations together, we have that in $G$, every $P_4$ is a subgraph of exactly one 5-cycle.  This is, of course, a simple observation, but it will be fundamental to the argument that follows.    

 \begin{theorem} \label{hsgtheorem} The Hoffman-Singleton graph is PRCF-bad.
\end{theorem}

\begin{proof} Let $G$ be the Hoffman-Singleton graph, and suppose to the contrary that a PRCF edge coloring $\varphi$ has been implemented on $G$.  Let $S$ be the set of all $P_4$'s in $G$, and partition $S$ into subsets $R$ and $T$ where $R$ contains all $P_4$'s whose three edges are colored with three different colors, and $T$ contains all $P_4$'s whose three edges are colored with only two colors.  In other words, with respect to $\varphi$, $R$ is the set of rainbow $P_4$'s and $T$ is the set of non-rainbow $P_4$'s.  Our goal is to now find lower and upper bounds on the order of $T$.

Consider a subgraph $H$ of $G$, where $H$ is isomorphic to the cycle $C_5$.  Two edges of $H$ must receive the same color under $\varphi$, and it follows that at least one of the five $P_4$'s that are subgraphs of $H$ is an element of $T$.  As previously seen, each $P_4$ in $G$ can be extended into exactly one 5-cycle in $G$.  This means that at least one out of every five elements of $S$ is an element of $T$, giving us a lower bound of $\left(\frac15\right)|S| \leq |T|$.

Now let $e = ab$ be an arbitrary edge of $G$.  As $G$ is 7-regular, let $c_1, \ldots, c_6$ be the other six edges of $G$ having endpoint $a$, and $d_1, \ldots, d_6$ be the other six edges of $G$ having endpoint $b$.  Without loss of generality, suppose $\varphi(e) = 0$ and $\varphi(c_i) = i$ for each $i \in \{1, \ldots, 6\}$ as is done in the figure below. 

\begin{center}
\includegraphics[scale=0.65]{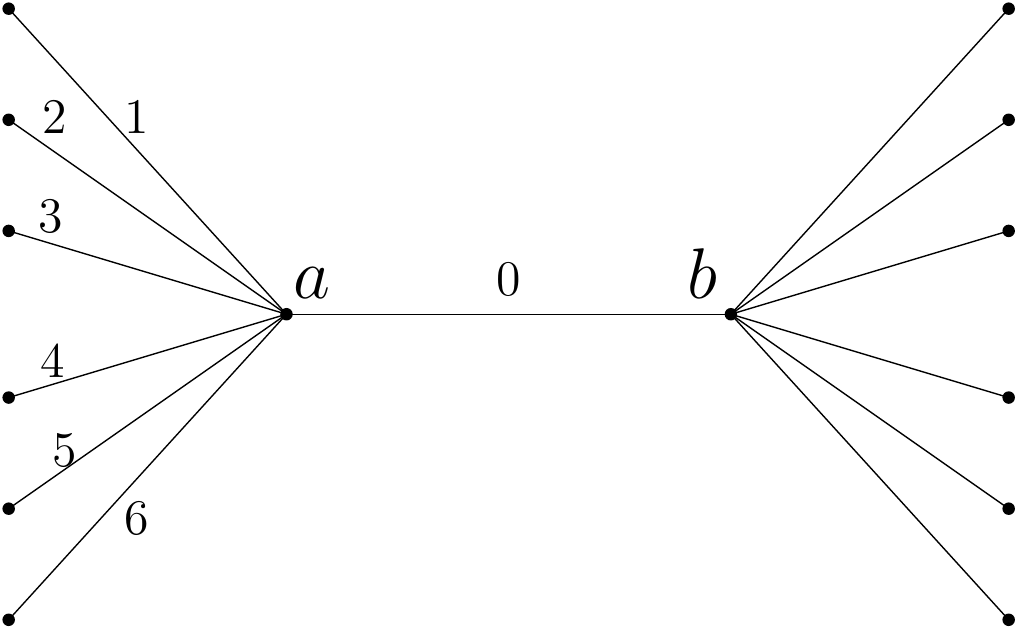}
\end{center}  

\noindent There are 36 distinct $P_4$'s in $G$ having $e$ as the ``middle" edge.  For each $i \in \{1, \ldots, 6\}$, there is at most one $j \in \{1, \ldots, 6\}$ such that $\varphi(d_j) = i$, which means that at most 6 of those 36 $P_4$'s are elements of $T$.  As $e$ was arbitrary, we have the upper bound $|T| \leq \left(\frac16\right)|S|$.

Combining the lower and upper bounds established for $|T|$, we have $\left(\frac15\right)|S| \leq |T| \leq \left(\frac16\right)|S| \implies \frac15 \leq \frac16$.  From this contradiction, we conclude that $G$ has no PRCF edge coloring.\qed
\end{proof} 

Define a PRCF-bad graph to be \textit{critical} if all of its proper subgraphs are PRCF-good.  We close in this section by observing that the Hoffman-Singleton graph is not critical.  Again, let $G$ denote the Hoffman-Singleton graph, and for some $v \in V(G)$, let $G' = G \setminus \{v\}$.  As is done in the proof of Theorem \ref{hsgtheorem}, suppose that a PRCF edge coloring $\varphi$ has been implemented on $G'$, and let $S'$ be the set of all $P_4$'s in $G'$, with $T' \subset S'$ the set of all $P_4$'s which are non-rainbow with respect to $\varphi$.  

As $G$ is a 7-regular graph of order 50, we have $|E(G)| = 175$, and since each edge of $G$ is the ``middle" edge of 36 distinct $P_4$'s, it follows that $|S| = 36(175) = 6300$.  It was previously seen that each $P_4$ of $G$ appears as a subgraph of exactly one 5-cycle of $G$, so $G$ contains $\frac{6300}{5} = 1260$ distinct 5-cycles.  It is known that $G$ is vertex-transitive, so from this, we conclude that, in $G$, vertex $v$ appears in $\frac{5(1260)}{50} = 126$ distinct 5-cycles.  This gives us a lower bound of $1260 - 126 = 1134 \leq |T'|$.  From the proof of Theorem \ref{hsgtheorem}, we had $|T| \leq \left(\frac16\right)|S| = \left(\frac16\right)\left(6300\right) = 1050$, and this upper bound for $|T|$ in $G$ is also an upper bound for $|T'|$ in $G'$.  Combining our bounds into one inequality, we have the contradiction $1134 \leq |T'| \leq 1050$. 

We call attention to the fact that the Hoffman-Singleton graph is not critically PRCF-bad, because ultimately, we were unable to decide whether or not there exists a critically PRCF-bad graph of girth 5.  It seems extremely unlikely that, in forming a critically PRCF-bad subgraph of the Hoffman-Singleton graph, one would be forced to delete so many edges and vertices that the resulting graph $H$ had girth $g(H) \geq 6$, and we would be interested in seeing an argument resolving this point.

\section{Generalized Polygons}

Graphs having diameter $d$ and girth $2d+1$ (like the Hoffman-Singleton graph where $d=2$) are called \textit{Moore graphs}.  This relationship between diameter and girth was particularly convenient for the proof of Theorem \ref{hsgtheorem} as it allowed us to immediately find a lower bound on $|T|$.  Moore graphs have been extensively studied in the literature (see \cite{miller} for a historical perspective), but unfortunately, they are quite rare.  In light of Theorem \ref{hsgtheorem}, we now desire to find PRCF-bad graphs of girth larger than 5, and the only Moore graphs fitting that bill are odd cycles $C_n$, which for $n > 3$ are PRCF-good.  We now instead turn our attention to the family of graphs known as \textit{generalized polygons}.

A generalized polygon is a graph $G$ where $diam(G) = d$ and $g(G) = 2d$.  In such graphs, a vertex $v$ is referred to as being \textit{thick} if $\deg(v) \geq 3$ and \textit{thin} otherwise.  A \textit{thick generalized polygon} is one in which every vertex is thick.  Like Moore graphs, generalized polygons are well-studied, and for background given from a graph-theoretic standpoint, we recommend \cite{godsil}.  Theorem \ref{feithigmantheorem} is the celebrated result of Feit and Higman showing the possible $d$ for which a thick generalized polygon exists.

\begin{theorem} \label{feithigmantheorem} Let $G$ be a thick generalized polygon of diameter $d$.  Then $d \in \{2,3,4,6,8\}$.  
\end{theorem}

We will focus our discussion on two generalized polygons.  The \textit{split Cayley hexagon} is bipartite, has diameter 6 and girth 12, and is regular of degree $r$ where $r-1$ can be taken to be any prime power.  The Ree-Tits octagon is bipartite with diameter 8 and girth 16.  It is semiregular, and letting $(A,B)$ be the bipartition of its set of vertices, all vertices in $A$ have degree $q+1$ and all vertices of $B$ have degree $q^2+1$ where $q$ can be taken as any odd power of 2.  In Theorems \ref{hexagon} and \ref{octagon}, we will show that each of these graphs, when $r$ and $q$ are made sufficiently large, are PRCF-bad.  Each proof will have the same outline as the proof of Theorem \ref{hsgtheorem}.

\begin{theorem} \label{hexagon} A sufficiently large split Cayley hexagon is PRCF-bad.
\end{theorem}

\begin{proof} Let $G$ be a split Cayley hexagon which is regular of degree $r$.  Assume that a PRCF edge coloring $\varphi$ has been implemented on $G$.  Consider a path $P_8$ with vertices labeled as below.

\begin{center}
\includegraphics[scale=0.65]{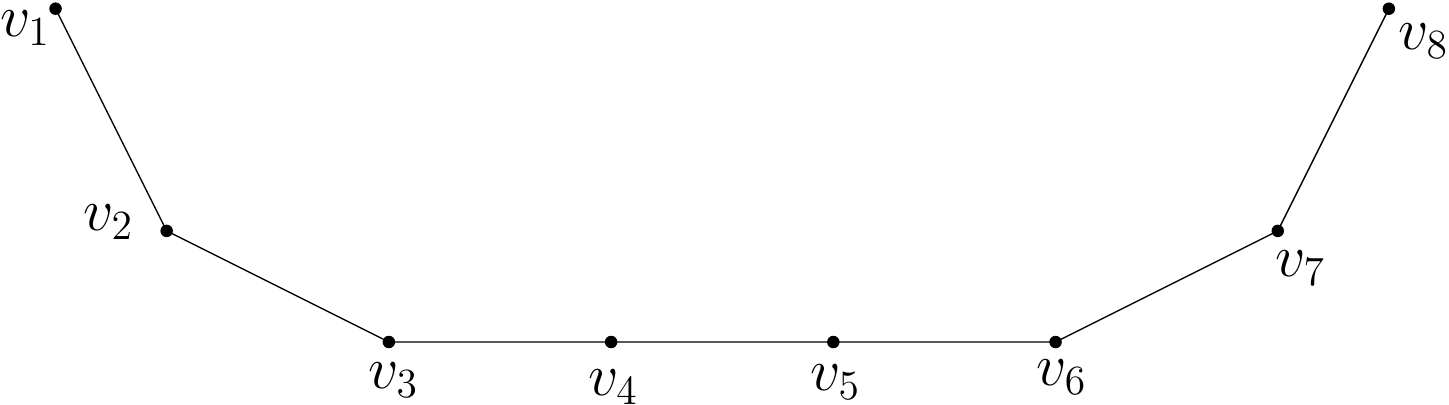}
\end{center}

Since $diam(G) = 6$, there exists a path $P$ in $G$ beginning at $v_8$, terminating at $v_1$, and having length less than or equal to 6.  Note that $P$ cannot contain any vertex $v_i$ for $i \in \{2, \ldots, 7\}$, and as well, $P$ cannot be of length less than or equal to 4, seeing as either of these characteristics would violate $g(G) = 12$.  Also, this path $P$ cannot have length 6 because $G$ is bipartite, and such a path would produce an odd cycle in $G$.  We conclude that $P$ has length exactly equal to 5.  Furthermore, it must be the case that $P$ is unique, as two distinct paths of length 5, each beginning at $v_8$ and terminating at $v_1$, would imply $g(G) \leq 10$.  From these observations, we conclude that every $P_8$ in $G$ is a subgraph of exactly one 12-cycle.

Let $S$ be the set of all $P_8$'s in $G$.  Let $R \subset S$ be the set of all $P_8$'s that are rainbow with respect to $\varphi$, and let $T = S \setminus R$ be the set of all $P_8$'s that are not rainbow.  Our goal is to now find lower and upper bounds for $\frac{|T|}{|S|}$ -- that is, the proportion of $P_8$'s in $G$ that are not rainbow.  As $\varphi$ is a PRCF edge coloring, each 12-cycle in $G$ must have at least two edges colored the same color.  For any two edges of $C_{12}$, there are at least two subgraphs of that $C_{12}$ which are isomorphic to $P_8$ and contain those two edges.  From the previous observation that every $P_8$ in $G$ is a subgraph of exactly one 12-cycle, we conclude that at least two out of every twelve $P_8$'s in $G$ are elements of $T$.  Thus $\frac16 \leq \frac{|T|}{|S|}$.

In the proof of Theorem \ref{hsgtheorem}, we found an upper bound for $|T|$ directly.  Here, it is easier to first determine $|S|$, and then find a lower bound for $|R|$, which  will suffice for our argument since $|T| = |S| - |R|$.  For any $v \in V(G)$, the number of paths having length 7 and beginning at $v$ is given by $r(r-1)^6$.  It follows that $|S| = \left(\frac12\right)|V(G)|r(r-1)^6$.  Note that the presence of the coefficient $\frac12$ is to avoid counting each $P_8$ twice.

To find a lower bound for $|R|$, again suppose $v \in V(G)$, and let us think about forming a rainbow $P_8$ starting at $v$.  There are $r$ neighbors of $v$, and any of them may be chosen to be on the path.  For any vertex $x_1$ in the open neighborhood $N(v)$, there are $r-1$ vertices in $N(x_1)\setminus\{v\}$ that may be chosen to continue the path.  For any vertex $x_2 \in N(x_1)\setminus\{v\}$, there are $r-1$ vertices in $N(x_2)\setminus\{x_1\}$, but it may be the case that for some $y \in N(x_2)\setminus\{x_1\}$, $\varphi(vx_1) = \varphi(x_2y)$.  This means that, from $x_2$, there are at least $r-2$ vertices to extend the rainbow path.  Continuing in this fashion, we find a lower bound $|R| \geq \left(\frac12\right)|V(G)|r(r-1)(r-2)(r-3)(r-4)(r-5)(r-6)$.  Note that the coefficient $\frac12$ prevents double counting, just as it did in our formula for $|S|$.

We now have the following upper bound for $|T|$.

\begin{equation}
|T| \leq |S| - \left(\frac12\right)|V(G)|r(r-1)(r-2)(r-3)(r-4)(r-5)(r-6)
\label{eq1}
\end{equation}

Now divide each side of Inequality \ref{eq1} by $|S|$ and simplify the right-hand side through use of our earlier expression for $|S|$.  Combining the result with our lower bound for $\frac{|T|}{|S|}$, we arrive at Inequality \ref{eq2}.

\begin{equation}
\frac16 \leq \frac{|T|}{|S|} \leq 1 - \frac{(r-2)(r-3)(r-4)(r-5)(r-6)}{(r-1)^5}
\label{eq2}
\end{equation} 

As $r \to \infty$, the expression on the right-hand side of Inequality \ref{eq2} goes to 0.  In constructing a split Cayley hexagon, our only restriction on the selection of $r$ is that $r-1$ must be a power of a prime, so we are free to select $r$ arbitrarily large and establish a contradiction.\qed
\end{proof}

In the statement of Theorem \ref{hexagon}, we did not say how large the order of a split Cayley hexagon must be to actually have the graph be PRCF-bad.  It turns out that the smallest value of $r$ which results in the contradiction being present in the proof above -- that is, for the right-hand side of Inequality \ref{eq2} to dip below $\frac16$ -- is $r = 90$.  The order of a 90-regular graph $G$ with $diam(G) = 6$ and $g(G) = 12$ is given by the expression $\displaystyle 1 + \sum_{i=0}^{5} 90(89)^i = 508,276,320,301$.  Quite large indeed!

We omit proofs showing that PRCF-bad graphs exist among the families of generalized polygons of girth 6 and 8 as they proceed similarly to the proof of Theorem \ref{hexagon}.  We will, however, formally show that there exists a PRCF-bad graph of girth 16.  Theorem \ref{octagon} will also follow the same outline as Theorem \ref{hexagon}, but since the Ree-Tits octagon is not regular, the counts that are involved are carried out slightly differently.

\begin{theorem} \label{octagon} A sufficiently large Ree-Tits octagon is PRCF-bad.  
\end{theorem}

\begin{proof} Let $G$ be a Ree-Tits octagon with bipartition $(A,B)$.  In other words, $V(G) = A \cup B$ with each vertex of $A$ having degree $q+1$ and each vertex of $B$ having degree $q^2 + 1$.  Assume $\varphi$ is a PRCF edge coloring of $G$.  Let $S$ be the set of all $P_{10}$'s in $G$, with $R, T$ being the sets of rainbow $P_{10}$'s and non-rainbow $P_{10}$'s, respectively, with respect to $\varphi$.  By the same reasoning displayed in the proof of Theorem \ref{hexagon}, each $P_{10}$ in $G$ is a subgraph of exactly one 16-cycle in $G$.  Since no cycle in $G$ is rainbow, any 16-cycle has at least two edges colored the same color under $\varphi$, and it follows that, for any 16-cycle in $G$, at least two of the sixteen $P_{10}$'s that are subgraphs of that 16-cycle are elements of $T$.  Hence $\frac18 \leq \frac{|T|}{|S|}$.

Note that any $P_{10}$ in $G$ has one of its two end vertices in $A$ and the other in $B$, and furthermore, any path must alternate between vertices of $A$ and vertices of $B$.  An expression for the order of $S$ is therefore given by $|S| = |A|(q+1)(q^2)(q)(q^2)(q)(q^2)(q)(q^2)(q) = |A|q^{12}(q+1)$.

As in the proof of Theorem \ref{hexagon}, we find an upper bound for $|T|$ by first finding a lower bound for $|R|$.  Here, we have $|R| \geq |A|(q+1)(q^2)(q-1)(q^2 - 2)(q-3)(q^2 - 4)(q-5)(q^2 - 6)(q-7)$.  Our upper bound for $|T|$ is given by Inequality \ref{eq3}.     

\begin{equation}
|T| \leq |S| - |A|(q+1)(q^2)(q-1)(q^2 - 2)(q-3)(q^2 - 4)(q-5)(q^2 - 6)(q-7)
\label{eq3}
\end{equation} 

After dividing both sides by $|S|$ and then simplifying, we may combine Inequality \ref{eq3} with our lower bound for $\frac{|T|}{|S|}$ to obtain Inequality \ref{eq4}.  

\begin{equation}
\frac16 \leq \frac{|T|}{|S|} \leq 1 - \frac{(q-1)(q^2-2)(q-3)(q^2-4)(q-5)(q^2-6)(q-7)}{q^{10}}
\label{eq4}
\end{equation}

The right-hand side of Inequality \ref{eq4} goes to 0 as $q \to \infty$, and since $q$ could be taken as any odd power of 2, we do indeed have a contradiction.\qed
\end{proof}

\section{Looking Ahead}

From the results of the previous two sections, we have that a PRCF-bad graph of girth $g$ exists for all $g \in \{3,4,5,6,8,12,16\}$.  We suppose that one could trivially fill in the gaps in this statement by, for each $n \in \{7, 9, 10, 11, 13, 14, 15\}$, letting $G$ be a disconnected graph with two components, one being a PRCF-bad Ree-Tits octagon and the other being the cycle $C_n$.  This feels like a cheat, though.  We are still drawn, however, to the question of determining whether or not PRCF-bad graphs exist of arbitrarily large girth.  In this section, we will put together a few scattered thoughts on how an attack on this problem may possibly proceed.    

With regard to the proofs of Theorems \ref{hsgtheorem}, \ref{hexagon}, and \ref{octagon}, the most convenient aspect of the Hoffman-Singleton graph and the generalized polygons was the fact that, in those respective graphs, all paths of a certain length could be uniquely extended to minimum cycles.  Unfortunately, there are no other Moore graphs or thick generalized polygons to look at as we try to find PRCF-bad graphs of girth greater than 16.  Other generalized polygons -- that is, those that are not thick -- are not candidates for consideration, which we will demonstrate in Theorem \ref{notthicktheorem}.  Its proof will use Lemma \ref{notthicklemma}, which can be seen with proof in \cite{godsil}.  Note that a \textit{$k$-fold subdivision} of a graph $G$ is a graph $G'$ formed by replacing each edge of $G$ with a path of length $k+1$.  One could visualize this process of $k$-fold subdivision as simply placing $k$ vertices of degree 2 on each edge of $G$.  For reference, a graph $G$ along with its 2-fold subdivision $G'$ is given in the figure below. 

\begin{center}
\includegraphics[scale=0.65]{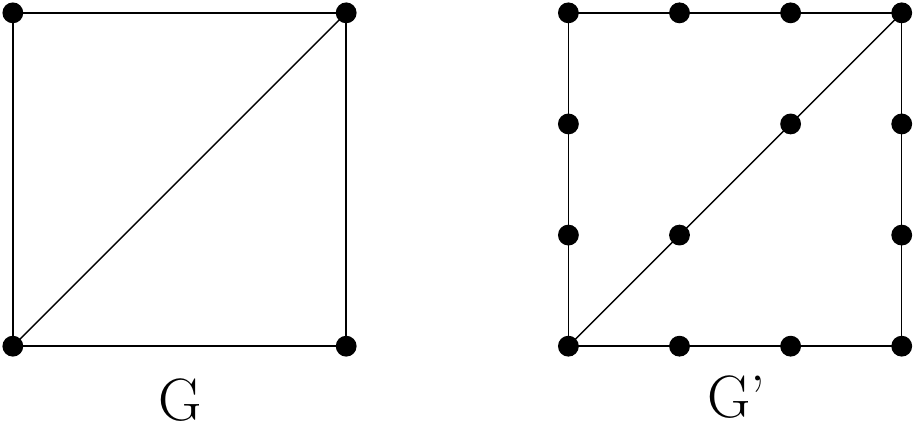}
\end{center}

\begin{lemma} \label{notthicklemma} A generalized polygon that is not thick is either an even cycle, the $k$-fold subdivision of a multiple edge, or the $k$-fold subdivision of a thick generalized polygon.
\end{lemma}

\begin{theorem} \label{notthicktheorem} Let $G$ be a generalized polygon that is not thick.  Then $G$ is PRCF-bad if and only if $G$ is isomorphic to the complete bipartite graph $K_{2,n}$ for some $n \geq 4$.  
\end{theorem}

\begin{proof} Lemma \ref{notthicklemma} categorizes those generalized polygons that are not thick.  Even cycles are PRCF-good.  Note that if $H$ is any graph, and $G$ is a $k$-fold subdivision of $H$ for some $k \geq 2$, then $G$ is guaranteed to be PRCF-good.  To see this, for each path of length $k+1$ in $G$ that replaced an edge of $H$, choose two adjacent vertices of degree 2, and color the edge containing them red.  In any proper edge coloring of $G$ which has those edges colored red, all cycles will have two or more red edges.

Now suppose that $G$ is a 1-fold subdivision of a thick generalized polygon $P$.  Let $diam(P) = d$ and $g(P) = 2d$.  Consider a breadth-first search tree of $P$ rooted at some $v \in V(P)$.  Label the levels of this tree as $L_0, \ldots, L_d$ where $L_0 = \{v\}$, $L_1 = N(v)$, $L_2 = N(L_1) \setminus \{v\}$, and so on.  For any edge $xy$ in $P$, let $z_{x,y} \in V(G)$ be the vertex of degree 2 that is adjacent to each of $x,y$.  Consider all edges of the form $z_{x,y}y$ where $x \in L_i$ and $y \in L_{i+1}$ with $i \in \{0, \ldots, d-2\}$.  Place a partial edge coloring on $G$ by coloring these edges all the same color, say $c_1$.  This partial coloring is proper as none of these edges are incident, and it guarantees that, should it be extended to a full edge coloring of $G$, no vertex of $L_0 \cup \cdots \cup L_{d-2}$ will be contained by a rainbow cycle in $G$.  We may then delete from $G$ all vertices of $L_0 \cup \cdots \cup L_{d-2}$ to form a new graph $G'$.  We then choose a new $v \in V(G')$ and repeat the above process, this time using color $c_2$.  Continuing in this manner, we see that no vertex of $G$ will appear in a rainbow cycle.  

Finally, note that if $G$ is a 1-fold subdivision of a multiple edge, $G$ is isomorphic to $K_{2,n}$.  It is seen in \cite{hoffmanetal} that $K_{2,n}$ is PRCF-bad if and only if $n \geq 4$.\qed

\end{proof}

So where to look now?  Ideally, we would be able to adapt the line of proof used in Theorems \ref{hsgtheorem}, \ref{hexagon}, and \ref{octagon} to find PRCF-bad graphs of larger girth, and it seems a good bet to focus our attention on the family of graphs known as cages.  An $(r,g)$-\textit{cage} is defined to be a graph of minimum order among all $r$-regular graphs of girth $g$.  Such graphs are shown to exist for all $r \geq 2$, $g \geq 3$ in \cite{sachs}.  In the literature, the order of an $(r,g)$-cage is typically denoted $n(r,g)$.  See \cite{exoosurvey} for an expansive history.

The method of proof used in Theorems \ref{hsgtheorem}, \ref{hexagon}, and \ref{octagon} essentially consisted of two parts.  Once it was assumed that a graph $G$ had been given a PRCF edge coloring $\varphi$, we defined $S$ to be the set of all paths $P_k$ in $G$ of a certain length, and let $T \subset S$ contain all non-rainbow paths.  If this method is to be generalized, note that $k$ would be equal to $\frac{g+3}{2}$ when $g$ is odd and $\frac{g+4}{2}$ when $g$ is even.  Our two tasks were to find lower and upper bounds for the order of $T$, eventually reaching a contradiction.  If $G$ is a cage, it is by definition regular, so an upper bound for the corresponding $|T|$ should be relatively easy to find.  When $G$ was a Moore graph or a generalized polygon as in Theorems \ref{hsgtheorem}, \ref{hexagon}, and \ref{octagon}, it was straightforward to find a lower bound for $|T|$, with the reason being that every path $P_k$ was a subgraph of exactly one minimum cycle in $G$.  If instead $G$ is a cage that is not a Moore graph or generalized polygon, this will not be the case.  Each $P_k$ will be a subgraph of at most one minimum cycle, but there will be some of these paths that are not subgraphs of any minimum cycles at all.  That said, since an $(r,g)$-cage has the smallest order among all $r$-regular graphs of girth $g$, it may be that such a graph $G$ ends up having a large enough ``threshold" of minimum cycles that, should we be able to count them, a contradiction may still present itself regarding the lower and upper bounds for $|T|$.         

Unfortunately, there are a number of issues which must be addressed for the above plan to work.  First of all, not many cages are actually known.  The value $n(r,g)$ is precisely given when there exists a Moore graph or a generalized polygon that is $r$-regular with girth $g$, and $n(r,g)$ has also been determined for a few selections of small $r,g$.  However, at this point in our search for PRCF-bad graphs of higher girth, we require $g \geq 17$, and we also need $r \geq g - 1$.  The reason for this latter condition is that a graph $G$ which is $r$-regular with $r \leq g - 2$ is automatically PRCF-good.  By Vizing's Theorem, $\chi'(G) = r \text{ or } r + 1$, and so $G$ could be properly colored with fewer than $g(G)$ colors, thus guaranteeing that no cycle is rainbow.  At present, determination of $n(r,g)$ for any $r \geq 16$, $g \geq 17$ seems beyond the ability of the mathematical community.  

In the literature, various upper bounds are given for $n(r,g)$ for general $r,g$.  The survey \cite{exoosurvey} gives a detailed compilation of many.  Our second source of difficulty comes from the fact that the methods used to describe the graphs realizing these upper bounds are algebraic and are typically non-constructive.  It seems quite difficult to count the number of minimum cycles in such graphs.  We leave these questions for future work. 

\begin{question} \label{q1} In terms of $r$ and $g$, what is the maximum number of $g$-cycles among all $(r,g)$-cages? 
\end{question} 

A precise answer to Question \ref{q1} seems well beyond our current reach, but even a lower bound on the number of distinct $g$-cycles that can be found in an $(r,g)$-cage may be of use in finding PRCF-bad girths of higher order.  

\begin{question} \label{q2} Suppose $G$ is an $(r,g)$-cage.  Let $S$ be the set of all paths $P_k$ in $G$ where $k = \frac{g+3}{2}$ if $g$ is odd and $k = \frac{g+4}{2}$ if $g$ is even.  Let $S_0 \subset S$ be the set of all $P_k$ that appear as a subgraph of some $g$-cycle in $G$.  As $r \to \infty$, is is true that $\frac{|S_0|}{|S|} \to b$ for some $b > 0$?   
\end{question}

We close by remarking that if Question \ref{q2} could be answered in the affirmative, it would likely confirm the existence of PRCF-bad graphs of arbitrarily large girth.



\begin{thebibliography}{99}


\bibitem{exoosurvey} G. Exoo and R. Jajcay, Dynamic cage survey, \textit{Electron. J. Combin.} \textbf{20} (2013), Dynamic Survey DS16.

\bibitem{feithigman} W. Feit and G. Higman, The nonexistence of certain generalized polygons, \textit{J. Algebra} \textbf{1} (1964), 114 -- 131.

\bibitem{godsil} C. D. Godsil and G. Royle, \textit{Algebraic Graph Theory}, Springer-Verlag, New York, 2001.

\bibitem{hoffmansingleton} A. J. Hoffman and R. R. Singleton, On Moore graphs with diameters 2 and 3, \textit{IBM J. Res. Develop.} \textbf{4} (1960), 497 -- 504.

\bibitem{hoffmanetal} D. G. Hoffman, P. Johnson, M. Noble, A. Owens, G. Puleo, and N. Terry, Proper rainbow-cycle-forbidding edge colorings of graphs, preprint.

\bibitem{miller} M. Miller and J. \u{S}ir\'{a}\u{n}, Moore graphs and beyond: A survey of the degree/diameter problem, \textit{Electron. J. Combin.} \textbf{20} (2013), Dynamic Survey DS14.

\bibitem{sachs} H. Sachs, Regular graphs with given girth and restricted circuits, \textit{J. London Math. Soc.} \textbf{38} (1963), 423 -- 429.

\bibitem{mathworld} E. W. Weisstein, ``Hoffman-Singleton Graph", from \textit{Mathworld} -- A Wolfram Web Resource.
http://mathworld.wolfram.com/Hoffman-SingletonGraph.html
 




\end{thebibliography}
\end{document}